\documentclass[a4paper,10pt]{article}
\usepackage{amsmath,amssymb,mathrsfs,epsfig,latexsym,makeidx,amsthm,float}
\theoremstyle{definition}
\newtheorem{theo}{Theorem}[section]

\newtheorem{coro}{Corollary}[section]
\newtheorem{defn}{Definition}[section]

\newtheorem{rem}{Remark}[section]

\newtheorem{prop}{Proposition}[section]

\usepackage{makeidx}
\usepackage{amssymb,amsmath}
\usepackage{graphics, graphicx}
\usepackage{pstricks,pst-node,pst-tree}
\usepackage{setspace}
\usepackage{geometry}
\usepackage{xcolor}
\newcommand\blfootnote[1]{%
  \begingroup
  \renewcommand\thefootnote{}\footnote{#1}%
  \addtocounter{footnote}{-1}%
  \endgroup
}

\title{\textbf{QUADRILATERAL LABYRINTH FRACTALS}}
\author{\textbf{Harsha Gopalakrishnan$^{1}$\blfootnote{The first author was supported by MHRD grant.} and Srijanani Anurag Prasad$^{2}$}}
\date{Department of Mathematics and Statistics\\Indian Institute of Technology, Tirupati\\$^{1}$\textit{harshagopal3@gmail.com}\\$^{2}$\textit{srijanani@iittp.ac.in}}
\begin{document}
	
\font\myfont=cmr12 at 20pt

\maketitle

\begin{flushleft}
\begin{abstract}
 In this paper, a class of fractals, called quadrilateral labyrinth fractals, are introduced and studied. They are a special kind of fractals on any quadrilateral on the plane. This type of fractal is motivated by labyrinth fractal on the unit square and triangle, which were already studied. This paper mainly deals with the construction of quadrilateral labyrinth fractal and studying its topological properties.
\end{abstract}
\end{flushleft}

\section{INTRODUCTION}
The labyrinth fractals were first introduced by Cristea and Steinsky[1,2] on a unit square in a plane and various properties like Hausdorff dimension, topological properties were also investigated in the same paper. Labyrinth fractals are generated from labyrinth sets that satisfy tree property, exit property, and corner property. It is observed that between any two points in horizontally and vertically blocked labyrinth fractals, there is a unique arc $a$ of infinite length, and the set of points where no tangents to $a$ exist is dense in $a$. It is also proved that there also exists labyrinth fractal with pairs of points, which are joined by an arc of finite length. A similar concept, namely, triangular labyrinth fractals were introduced and studied by Cristea and Paul Surer[3]. The construction of fractal is based on two triangular pattern systems, and correspondingly, two fractals are obtained. The triangular labyrinth fractals are classified and studied based on the length and nature of arcs between two points in the fractal.  In both triangle case and square case, labyrinth fractals are self-similar dendrites. 
\par
The  concepts of labyrinth fractals on the unit square were extended further in various aspects, such as mixed labyrinth fractals and supermixed labyrinth fractals[4,5]. Mixed labyrinth fractals are constructed by using a sequence of labyrinth patterns. i.e., at different stages of iteration, different patterns are used. Supermixed labyrinth fractals use a finite collection of labyrinth pattern at each stage of iteration. Mixed and supermixed labyrinth fractals are generally not self-similar. The rectangular labyrinth fractals were also introduced and studied in [4].
\par
 The classical labyrinth fractal on a square, introduced by Cristea and Stiensky, have a numerous applications in Physics, such as in planar nanostructures\cite{9}, in the construction of prototypes of ultra-wide band radar antennas\cite{10} and the fractal reconstruction of complicated images, signals and radar backgrounds\cite{11}. Fractal labyrinths are also used in combination with genetic algorithms for the synthesis of big robust antenna arrays and nanoantennas in telecommunication\cite{12}.
\par
The labyrinth fractal can be generalized further to any convex quadrilateral in a   
            plane. Unlike the construction of labyrinth fractal in square and triangle, quadrilateral 
            lacks self-similarity when divided into smaller quadrilaterals. So, this class of fractal  
            cannot have self-similarity. Hence, it is interesting to study the labyrinth fractal on a  convex 
            quadrilateral with the properties precisely the same as the square and triangle  
            case but lacks similarity at each stage of construction. Topological properties of these fractals are also studied in this paper.
\par The paper begins with the construction of a fractal on any quadrilateral. The division to smaller quadrilaterals is motivated from the construction of both triangular and square labyrinths. Here instead of similar quadrilaterals, the obtained set is a set of similar quadrilaterals along with a set of two kinds of parallellograms at each stage of construction. Some preliminary results are also proved in this section.
\par The next section of this paper is devoted to the definition of labyrinth fractals on convex quadrilaterals and their basic properties. The definition coincides exactly with the previous definitions on squares and triangles. i.e., the quadrilateral labyrinth satisfies the tree property, exit property and corner property. The  topological properties are studied in section 4. The study deals with the topological properties of complements of labyrinths on the quadrilaterals, besides the topological properties of labyrinth fractal itself.
\par
The labyrinth fractal in convex quadrilateral can be applied in real-life where there is a lack of self-similarity. In any real life situation, it is not necessary that the basic structure of the study is  always square-shaped or  triangle-shaped. In such a case, the labyrinth fractal on the quadrilateral can be used as it deals with any convex and four-sided polygon in a plane.

\section{CONSTRUCTION OF FRACTAL}
In this section, a fractal on any convex quadrilateral on a plane is constructed. Some preliminary results regarding the construction are also proved. These results will in turn helps in proving various theorems in upcoming sections.
\par
Consider a convex quadrilateral $Q$ on a plane. Divide it into two triangles using the smallest diagonal of the quadrilateral (If it has equal diagonal then choose anyone of them). This diagonal is called the common side for the obtained triangles. The endpoints of the common side are called common points for the obtained triangles. Name the quadrilateral as follows: Starting from any of the common point, move anticlockwise and name the vertices of $Q$ as $Q_1,Q_2,Q_3$ and $Q_4$ respectively.  
 Then $Q=Q_1Q_2Q_3Q_4$ and triangles are $\Delta_1=Q_1Q_2Q_3$ and $\Delta_2=Q_3Q_4Q_1$.
\begin{figure}[tb]
\begin{center}
\includegraphics[width=0.4\linewidth]{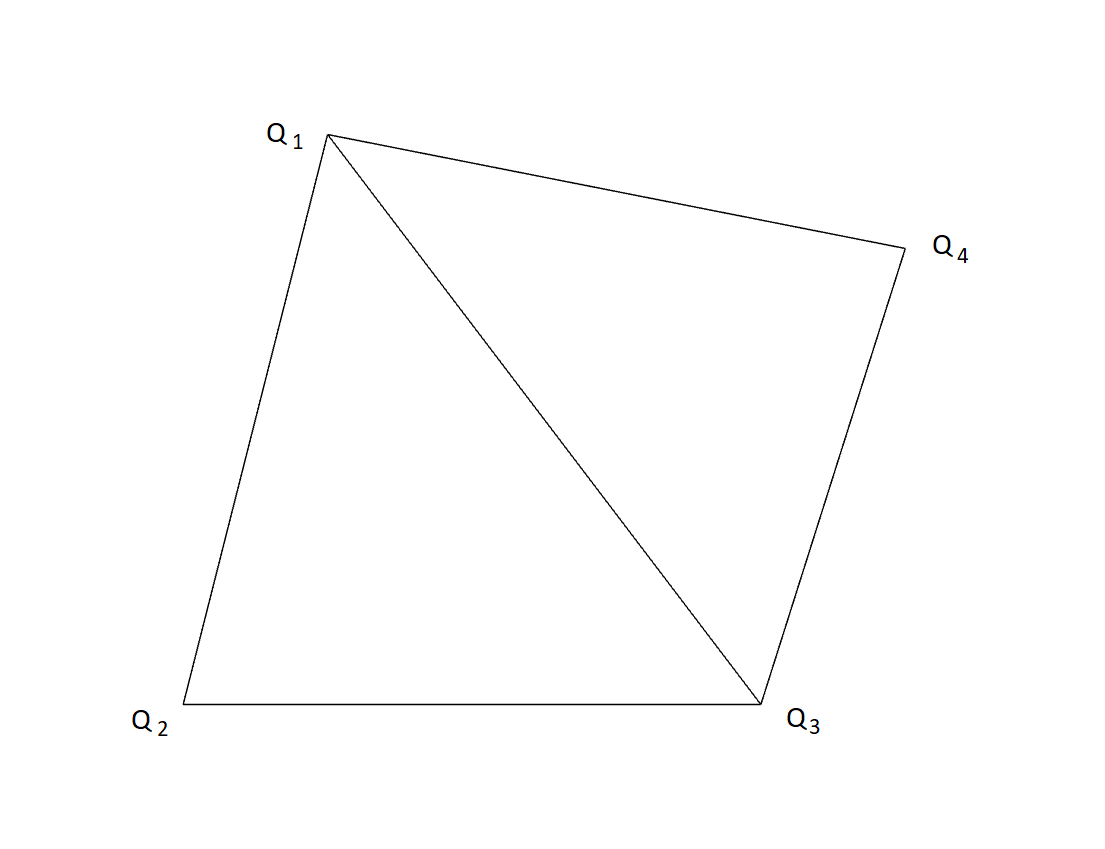}
\end{center}
\caption{This is an example of the naming of a quadrilateral}
\label{firstfig}
\end{figure}

Let $x\in Q$ then either $x\in\Delta_1$ or $x\in\Delta_2$. If $x\in\Delta_1$, and $x=\alpha_1Q_1+\alpha_2Q_2+\alpha_3Q_3$ is the unique representation in $\Delta_1$,
then $x=\alpha_1Q_1+\alpha_2Q_2+\alpha_3Q_3+0Q_4$ is unique representation in $Q$.
If $x\in\Delta_2$, and $x=\alpha_3Q_3+\alpha_4Q_4+\alpha_1Q_1$ is the unique representation in $\Delta_2$,
then $x=\alpha_1Q_1+0Q_2+\alpha_3Q_3+\alpha_4Q_4$ is unique representation in $Q$.
Now if $x$ lies in the line $Q_1Q_3$, then the above two definitions coincide. Hence any $x\in Q$ is uniquely represented as $(\alpha_1,\alpha_2,\alpha_3,\alpha_4)$, where $\alpha_i$ is the coefficient of $Q_i$, for $i=1,2,3,4$ as in the above representation with either $\alpha_2=0$ or $\alpha_4=0$ and $(\alpha_1,\alpha_2,\alpha_3,\alpha_4)$ is called the unique representation of $x$ in $Q$.
If $Q'=R_1R_2R_3R_4$ is a convex quadrilateral inside $Q$, define a map $P_{Q'}:Q \to Q'$ as \\$P_{Q'}(x)=\alpha_1R_1+\alpha_2R_2+\alpha_3R_3+\alpha_4R_4$, where $(\alpha_1,\alpha_2,\alpha_3,\alpha_4)$ is the unique representation of $x$ in $Q$.
\par Let $X$ and $Y$ are two topological spaces. A homeomorphism from $X$ to $Y$ is a bijection $f:X\to Y$ such that both $f$ and $f^{-1}$ are continuous.

\begin{prop}
Let $Q=Q_1Q_2Q_3Q_4$ be a convex quadrilateral and $Q'=R_1R_2R_3R_4$ is a convex quadrilateral inside $Q$. Then the map $P_{Q'}:Q\to Q'$ given by\\ $P_{Q'}(x)=\alpha_1R_1+\alpha_2R_2+\alpha_3R_3+\alpha_4R_4$, where $(\alpha_1,\alpha_2,\alpha_3,\alpha_4)$ is the unique representation of $x$ in $Q$, is a homeomorphism between $Q$ and $Q'$. 
\end{prop}
\begin{proof}
Let $P_{Q'}(x)=P_{Q'}(y)$ for $x,y\in Q$ with $(\alpha_1,\alpha_2,\alpha_3,\alpha_4), (\beta_1,\beta_2,\beta_3,\beta_4)$ as the unique representation of $x$ and $y$ in $Q$ respectively.\\
Case 1: $x\in\Delta_1=Q_1Q_2Q_3$ and $y\in\Delta_2=Q_1Q_3Q_4$.\\
Here the representation of $x$ and $y$ are $x=(\alpha_1,\alpha_2,\alpha_3,0)$ and $y=(\beta_1,0,\beta_3,\beta_4)$ in $Q$.
Hence $P_{Q'}(x)=P_{Q'}(y)\implies (\alpha_1,\alpha_2,\alpha_3,0)=(\beta_1,0,\beta_3,\beta_4)$ in $Q'$.
Thus $\alpha_1=\beta_1$, $\alpha_2=0$, $\alpha_3=\beta_3$, $\beta_4=0$ and so $x=y$ and lies on the line joining $Q_1$ and ~$Q_3$.\\
Case 2: $x,y\in\Delta_1=Q_1Q_2Q_3$ or $x,y\in\Delta_2=Q_3Q_4Q_1$\\
Then $x=(\alpha_1,\alpha_2,\alpha_3,0)$ and $y=(\beta_1,\beta_2,\beta_3,0)$ or $x=(\alpha_1,0,\alpha_3,\alpha_4)$ and \\$y=(\beta_1,0,\beta_3,\beta_4)$ in $Q$.
$P_{Q'}(x)=P_{Q'}(y)\implies (\alpha_1,\alpha_2,\alpha_3,0)=(\beta_1,\beta_2,\beta_3,0)$ or $(\alpha_1,0,\alpha_3,\alpha_4)=(\beta_1,0,\beta_3,\beta_4)$ in $Q'$.
Either case $\alpha_i=\beta_i$ for $i=1,2,3,4$ and so $x=y$.\\
Thus $P_{Q'}$ is injective. Surjectivity follows from the definition of $P_{Q'}$ and hence $P_{Q'}$ is bijection.
Now, let $(x_n)=\alpha_{1n}Q_1+\alpha_{2n}Q_2+\alpha_{3n}Q_3+\alpha_{4n}Q_4$ be a sequence in $Q$ which converges to $x=\alpha_1Q_1+\alpha_2Q_2+\alpha_3Q_3+\alpha_4Q_4$ in $Q$. 
Then $(\alpha_{1n},\alpha_{2n},\alpha_{3n},\alpha_{4n})$ is a sequence in $\mathbb{R}^4$ converging to $(\alpha_1,\alpha_2,\alpha_3,\alpha_4)$ in $\mathbb{R}^4$. Hence $(\alpha_{in})\to\alpha_i$ for $i=1,2,3,4$.
Thus $(\alpha_{1n}R_1+\alpha_{2n}R_2+\alpha_{3n}R_3+\alpha_{4n}R_4)\to \alpha_1R_1+\alpha_2R_2+\alpha_3R_3+\alpha_4R_4$\\
$i.e,$ $P_{Q'}(x_n)\to P_{Q'}(x)$ and so $P_{Q'}$ is continuous. In a similar way $(P_{Q'})^{-1}$ is also continuous.
Thus $P_{Q'}$ is a homeomorphism.
\end{proof}
\begin{flushleft}
For $m\ge 2$, consider the sets
\end{flushleft}
$A_1=\{(k_1,k_2,k_3,k_4)\in(\mathbb{N}\cup\{0\})^4 : k_1+k_2+k_3=m-1$, $k_4=0$ and $k_2\neq 0\}$\\
$A_2=\{(k_1,k_2,k_3,k_4)\in(\mathbb{N}\cup\{0\})^4 : k_1+k_3+k_4=m-1$, $ k_2=0$ and $k_4\neq 0\}$\\
$A_3=\{(k_1,k_2,k_3,k_4)\in(\mathbb{N}\cup\{0\})^4 : k_1+k_3=m-1$ and $k_2= k_4=0\}$\\
And let
\begin{equation}
 A=A_1\cup A_2\cup A_3
\end{equation}
Define a function $S_m$ on $A$ as follows:\\
$S_m(k_1,k_2,k_3,k_4)=R_1R_2R_3R_4$, where,\\
$R_1= \frac{(k_1+1)Q_1+k_2Q_2+k_3Q_3+k_4Q_4}{m} $;
$R_2= \frac{k_1Q_1+(k_2+1)Q_2+k_3Q_3+k_4Q_4}{m} $;\\
$R_3=\frac{k_1Q_1+k_2Q_2+(k_3+1)Q_3+k_4Q_4}{m}$;
$R_4=\frac{(k_1+1)Q_1+(k_2-1)Q_2+(k_3+1)Q_3+k_4Q_4}{m}$; 
\\if $(k_1,k_2,k_3,k_4)\in A_1$. And,\\
$R_1= \frac{(k_1+1)Q_1+k_2Q_2+k_3Q_3+k_4Q_4}{m} $;
$R_2= \frac{(k_1+1)Q_1+k_2Q_2+(k_3+1)Q_3+(k_4-1)Q_4}{m}$;\\
$R_3=\frac{k_1Q_1+k_2Q_2+(k_3+1)Q_3+k_4Q_4}{m}$;
$R_4=\frac{k_1Q_1+k_2Q_2+k_3Q_3+(k_4+1)Q_4}{m}$;\\ if $(k_1,k_2,k_3,k_4)\in A_2$.
And,\\
$R_1= \frac{(k_1+1)Q_1+k_2Q_2+k_3Q_3+k_4Q_4}{m} $;
$R_2= \frac{k_1Q_1+(k_2+1)Q_2+k_3Q_3+k_4Q_4}{m} $;\\
$R_3=\frac{k_1Q_1+k_2Q_2+(k_3+1)Q_3+k_4Q_4}{m}$;
$R_4=\frac{k_1Q_1+k_2Q_2+k_3Q_3+(k_4+1)Q_4}{m}$;\\ if $(k_1,k_2,k_3,k_4)\in A_3$.\\
Then $S_m(k_1,k_2,k_3,k_4)$ are parellellograms in $Q$ for $(k_1,k_2,k_3,k_4)\in A_1\cup A_2$ and $S_m(k_1,k_2,k_3,k_4)$ are quadrilaterals in $Q$, which are similar to $Q$ for $(k_1,k_2,k_3,k_4)\in\nolinebreak A_3$.
Let $\mathcal{S}_m=S_m(A_1)\cup S_m(A_2)\cup S_m(A_3)$.
The elements of the set\\ $C=\{S_m(m-1,0,0,0),S_m(0,m-1,0,0),S_m(0,0,m-1,0)S_m(0,0,0,m-1)\}\subseteq \mathcal S_m$ are called the corner quadrilaterals.
The elements $S_m(k_1,k_2,k_3,k_4)$ in $\mathcal S_m$ such that atmost two of the $k_i$'s are nonzero are said to be border quadrilaterals.
\par
Now choose a subset $\mathcal{W}_1$ of $\mathcal{S}_m$, which is called the set of white quadrilaterals of order $1$ and $\mathcal{B}_1=\mathcal{S}_m\setminus\mathcal{W}_1$ is called the set of black quadrilaterals of order $1$. 
For $n\ge 2$, the set of white quadrilaterals of order $n$ is given by,\\
$\mathcal{W}_n=\{P_{W_{n-1}}(W_1) : W_1\in\mathcal{W}_1,W_{n-1}\in\mathcal{W}_{n-1}\}$. Then $\mathcal{W}_n\subset\mathcal{S}_{m^n}$ and the set of black squares of order $n$ is given by $\mathcal{B}_n=\mathcal{S}_{m^n}\setminus\mathcal{W}_n$.
For $n\ge 1$, define \\$L_n=\bigcup\limits_{W\in\mathcal{W}_n}{W}$. Since each $L_n$ is closed and bounded, $L_n$ is compact. Thus $\{L_n\}$ is a monotonically decreasing sequence of compact sets. 
Define $\mathcal L_{\infty}=\bigcap\limits_{n=1}^{\infty}L_n$, which is the fractal set in $Q$.
\begin{prop}
Let $Q=Q_1Q_2Q_3Q_4$ be a convex quadrilateral and $\mathcal{S}_m$ is defined as above for any $m\ge 2$. Then $P_{Q'}\circ P_{Q''}=P_S$ in $Q$, where $S=P_{Q'}(Q'')$ and $Q''\in \mathcal S_n$ for any $n\ge m$.
\end{prop} 
\begin{proof}
Let $Q'=R_1R_2R_3R_4$ and $Q''=T_1T_2T_3T_4$.
Let $T_i=\alpha_1^iQ_1+\alpha_2^iQ_2+\alpha_3^iQ_3+\alpha_4^iQ_4$ be the unique representation of $T_i$ in $Q$ for $i=1,2,3,4$.
Choose $x\in Q$. WLOG suppose that $x\in\Delta_1$ and let $x=\beta_1Q_1+\beta_2Q_2+\beta_3Q_3+0Q_4$ be the unique representation of $x$ in $Q$.\\
\textbf{Case 1:}$Q''$ in $\Delta_1$.
In this case $\alpha_4^i=0$ for all $i=1,2,3,4$.
 Then,
\begin{equation*}
\begin{split}
P_{Q'}\circ P_{Q''}(x)&=P_{Q'}\circ P_{Q''}(\beta_1Q_1+\beta_2Q_2+\beta_3Q_3)\\
&=P_{Q'}(\beta_1T_1+\beta_2T_2+\beta_3T_3)\\
&=P_{Q'}(\beta_1(\alpha_1^1Q_1+\alpha_2^1Q_2+\alpha_3^1Q_3)+\beta_2(\alpha_1^2Q_1+\alpha_2^2Q_2+\alpha_3^2Q_3)+\beta_3(\alpha_1^3Q_1+\alpha_2^3Q_2+\alpha_3^3Q_3))\\
&=P_{Q'}((\beta_1\alpha_1^1+\beta_2\alpha_1^2+\beta_3\alpha_1^3)Q_1+(\beta_1\alpha_2^1+\beta_2\alpha_2^2+\beta_3\alpha_2^3)Q_2+(\beta_1\alpha_3^1+\beta_2\alpha_3^2+\beta_3\alpha_3^3)Q_3)\\
&=(\beta_1\alpha_1^1+\beta_2\alpha_1^2+\beta_3\alpha_1^3)R_1+(\beta_1\alpha_2^1+\beta_2\alpha_2^2+\beta_3\alpha_2^3)R_2+(\beta_1\alpha_3^1+\beta_2\alpha_3^2+\beta_3\alpha_3^3)R_3
\end{split}
\end{equation*}
Note that in the above step we can apply $P_{Q'}$ since sum of coefficients of $Q_i's$ are ~$1$ and coefficient of $Q_4$ is zero. The vertices of $S$ are\\
\begin{equation*}
\begin{split}
S_1= P_{Q'}(T_1)=P_{Q'}(\alpha_1^1Q_1+\alpha_2^1Q_2+\alpha_3^1Q_3)=\alpha_1^1R_1+\alpha_2^1R_2+\alpha_3^1R_3\\
S_2= P_{Q'}(T_2)=P_{Q'}(\alpha_1^2Q_1+\alpha_2^2Q_2+\alpha_3^2Q_3)=\alpha_1^2R_1+\alpha_2^2R_2+\alpha_3^2R_3\\
S_3= P_{Q'}(T_3)=P_{Q'}(\alpha_1^3Q_1+\alpha_2^3Q_2+\alpha_3^3Q_3)=\alpha_1^3R_1+\alpha_2^3R_2+\alpha_3^3R_3 \\
S_4= P_{Q'}(T_4)=P_{Q'}(\alpha_1^4Q_1+\alpha_2^4Q_2+\alpha_3^4Q_3)=\alpha_1^4R_1+\alpha_2^4R_2+\alpha_3^4R_3
\end{split}
\end{equation*}
Thus,
\begin{equation*}
\begin{split} 
P_{S}(x)&=P_S(\beta_1Q_1+\beta_2Q_2+\beta_3Q_3)\\
&=\beta_1S_1+\beta_2S_2+\beta_3S_3\\
&=\beta_1(\alpha_1^1R_1+\alpha_2^1R_2+\alpha_3^1R_3)+\beta_2(\alpha_1^2R_1+\alpha_2^2R_2+\alpha_3^2R_3)+\beta_3(\alpha_1^3R_1+\alpha_2^3R_2+\alpha_3^3R_3)\\
&=(\beta_1\alpha_1^1+\beta_2\alpha_1^2+\beta_3\alpha_1^3)R_1+(\beta_1\alpha_2^1+\beta_2\alpha_2^2+\beta_3\alpha_2^3)R_2+(\beta_1\alpha_3^1+\beta_2\alpha_3^2+\beta_3\alpha_3^3)R_3\\
&=P_{Q'}\circ P_{Q''}(x)
\end{split}
\end{equation*}
\textbf{Case 2:}$Q''$ in $\Delta_2$.
In this case $\alpha_2^i=0$ for all $i=1,2,3,4$ and $P_{Q'}\circ P_{Q''}(x)=P_S(x)$ follows same as in case $1$. \\
\textbf{Case 3:} $Q''$ along diagonal.
Here $\alpha_2^1=\alpha_4^1=0$, $\alpha_2^3=\alpha_4^3=0$, $\alpha_4^2=0$ and $\alpha_2^4=0$.\\ In this case also,  $P_{Q'}\circ P_{Q''}(x)=P_S(x)$ follows same as in case $1$.\\
In each case  $P_{Q'}\circ P_{Q''}(x)=P_S(x) $.
Hence  $P_{Q'}\circ P_{Q''}=P_S$
\end{proof}
\begin{rem}
In Proposition $2.2$, the position of $Q'$ does not matter in the calculation.
\end{rem}
\begin{rem}
Note that Proposition $2.2$ is valid only if $Q''\in \mathcal S_m$. \\
Let  $Q=Q_1Q_2Q_3Q_4$, where $Q_1=(0,1)$, $Q_2=(0,0)$, $Q_3=(1,0)$ and $Q_4=(1,1)$\\
and $Q'=R_1R_2R_3R_4$, where $R_1=(0,\frac{1}{4})$, $R_2=(0,0)$, $R_3=(\frac{1}{4},0)$ and $R_4=(\frac{1}{4},\frac{1}{4})$\\
and $Q''=T_1T_2T_3T_4$, where $T_1=(\frac{2}{4},\frac{3}{4})$, $T_2=(\frac{1}{4},\frac{1}{4})$, $T_3=(\frac{3}{4},\frac{1}{4})$ and $T_4=(\frac{3}{4},\frac{3}{4})$\\
Then,
\begin{equation*}
\begin{split}
T_1=\frac{1}{2}(0,1)+0(0,0)+\frac{1}{4}(1,0)+\frac{1}{4}(1,1)\\
T_2=\frac{1}{4}(0,1)+\frac{1}{2}(0,0)+\frac{1}{4}(1,0)+0(1,1)\\
T_3=\frac{1}{4}(0,1)+0(0,0)+\frac{3}{4}(1,0)+0(1,1)\\
T_4=\frac{1}{4}(0,1)+0(0,0)+\frac{1}{4}(1,0)+\frac{1}{2}(1,1)\\
\end{split}
\end{equation*}
Let $x=(\frac{1}{2},\frac{1}{4})=\frac{1}{4}(0,1)+\frac{1}{4}(0,0)+\frac{1}{2}(1,0)+0(1,1)\in Q$\\
Hence,
\begin{equation*}
\begin{split}
P_{Q'}\circ P_{Q''}(x)&=P_{Q'}(P_{Q''}(\frac{1}{4}(0,1)+\frac{1}{4}(0,0)+\frac{1}{2}(1,0)+0(1,1)))\\
&=P_{Q'}(\frac{1}{4}(\frac{2}{4},\frac{3}{4})+\frac{1}{4}(\frac{1}{4},\frac{1}{4})+\frac{1}{2}(\frac{3}{4},\frac{1}{4})+0(\frac{3}{4},\frac{3}{4}))\\
&=P_{Q'}(\frac{5}{16}(0,1)+\frac{1}{8}(0,0)+\frac{1}{2}(1,0)+\frac{1}{16}(1,1))
\end{split}
\end{equation*}
Note that  R.H.S is not well defined. This is because $Q''\not\in \mathcal S_m$.
\end{rem}
\section{QUADRILATERAL LABYRINTH FRACTALS}
In this section, a set of white quadrilaterals, called the labyrinth set, is chosen with some conditions. The corresponding fractal generated from this set is said to be quadrilateral labyrinth fractal. Besides that, some theorems regarding the labyrinth set and labyrinth fractals are also proved. Some examples of labyrinth sets are also given. This section requires some basic concepts in graph theory, which are included at the beginning of this section. 
\par
A graph $\mathcal{G}=(\mathcal{V}(\mathcal{G}),\mathcal{E}(\mathcal{G}))$ consists of a set of vertices $\mathcal{V}(\mathcal{G})$ and a set of edges $\mathcal{E}(\mathcal{G})$, where $\mathcal{E}(\mathcal{G})$ is a subset of the unordered pairs of $\mathcal{V}(\mathcal{G})$. If $\{v_1,v_2\}\in\mathcal{E}(\mathcal{G})$ then the two vertices $v_1$ and $v_2$ are said to be adjacent. A path is a sequence of pairwise distinct vertices $v_1,v_2,...,v_k$, $k\ge 1$ such that for every $j\in\{1,2,...,k-1\}$, the vertices $v_j$ and $v_{j+1}$ are adjacent. The vertices $v_1$ and $v_k$ are called initial and terminal vertices of the path respectively. If there exist an edge that connects initial and terminal vertex of a path, then the path is called a cycle (provided that  $k>2$). A graph is  said to be connected if there exist a path between any two vertices. A connected graph having no cycle is said to be a tree.
\par
For $n\ge 1$, the graph of $\mathcal{W}_n$ is defined as the graph with vertex set $\mathcal{V}(\mathcal{G}(\mathcal{W}_n))$ as the set of white quadrilaterals in $\mathcal{W}_n$ and the edge set  $\mathcal{E}(\mathcal{G}(\mathcal{W}_n))$ as the unordered pair of white quadrilaterals in $\mathcal{W}_n$, that have a common side. Such a graph is denoted by  $\mathcal{G}(\mathcal{W}_n)=\mathcal{G}(\mathcal{V}(\mathcal{G}(\mathcal{W}_n)),\mathcal{E}(\mathcal{G}(\mathcal{W}_n)))$. 
\par
For $n\ge 1$, the graph of $\mathcal{B}_n$ is defined as the graph with vertex set $\mathcal{V}(\mathcal{G}(\mathcal{B}_n))$ as the set of black quadrilaterals in $\mathcal{B}_n$ and the edge set  $\mathcal{E}(\mathcal{G}(\mathcal{B}_n))$ as the unordered pair of black quadrilaterals in $\mathcal{W}_n$, that have a common side or a common vertex. Such a graph is denoted by  $\mathcal{G}(\mathcal{B}_n)=\mathcal{G}(\mathcal{V}(\mathcal{G}(\mathcal{B}_n)),\mathcal{E}(\mathcal{G}(\mathcal{B}_n)))$. 
\begin{defn}
Let $m\ge 4$ and $\mathcal W_1\subseteq \mathcal S_m$. Then $\mathcal{W}_1$ is an $m\times m$ - quadrilateral labyrinth set if it satisfies the following properties:
\begin{enumerate}
\item \textbf{Tree Property:} $\mathcal G(\mathcal W_1)$ is a tree.
\item \textbf{Exit Property:} There exist exactly one $(k_1,k_2,0,0)\in A$ such that\\ $S_m(k_1,k_2,0,0)\in\mathcal W_1$ and $S_m(0,0,k_2,k_1)\in\mathcal W_1$ and exactly one\\ $(k_1,0,0,k_4)\in A$ such that  $S_m(k_1,0,0,k_4)\in W_1$ and $S_m(0,k_1,k_4,0)\in W_1$, where $A$ is given as in $(2.1)$. In this case  $S_m(k_1,k_2,0,0)$ is called the left exit, $S_m(0,0,k_1,k_2)$ is the right exit, $S_m(k_1,0,0,k_4)$ is the top exit and $S_m(0,k_1,k_4,0)$ is the bottom exit.
\item  \textbf{Corner Property:} If there is a white quadrilateral in $\mathcal W_1$ containing any of the vertex $Q_i (i=1,2,3,4)$ of $Q$, then the white quadrilateral in $\mathcal{S}_m$ containing the diagonally opposite vertex of $Q_i$ should not be in $\mathcal{W}_1$. i.e. $\mathcal{W}_1$ contains atmost one element from each of the sets\\
$ \{S_m(m-1,0,0,0),S_m(0,0,m-1,0)\}$ and  $\{S_m(0,m-1,0,0),S_m(0,0,0,m-~1)\}$.
\end{enumerate}
\end{defn}
\begin{defn}
If $\mathcal W_1\subseteq \mathcal S_m$ is a quadrilateral labyrinth set in the quadrilateral $Q$, the left side of $Q$ is defined as the side of $Q$ which contains the left exit. Right, bottom and top sides of $Q$ are defined analogously. The vertex of $Q$ at the intersection of top and left side is called top-left vertex. Top-right vertex, bottom-left vertex and bottom-right vertex are defined analogously. The corner quadrilateral containing top-left vertex is called top-left corner. Analogously, top-right corner, bottom-left corner and bottom-right corner are defined. 
\end{defn}

\begin{prop}
If $\mathcal{W}_1$ is a quadrilateral labyrinth set, no corner can be an exit of two adjacent sides. 
\end{prop}
\begin{proof}
Suppose the result is not true. WLOG suppose the bottom-left corner is an exit for both bottom side and left side. Then by the exit property, bottom-right corner is the right exit and the top-left corner is the top exit and so they belongs to $\mathcal W_1$. It contradicts the corner property. Thus no corner can be an exit of two adjacent sides. 
\end{proof}
\begin{prop}
If $\mathcal{W}_1$ is an $m\times m$ - quadrilateral labyrinth set, then $\mathcal{W}_n$ is an $m^n\times m^n$ - quadrilateral labyrinth set for all $n\ge 1$. 
\end{prop}
\begin{proof}
The proof is by induction. For $n=1$, it is clear that $\mathcal W_1$ is an $m\times m$~ - quadrilateral labyrinth set. Suppose $\mathcal W_{n-1}$ is an $m^{n-1}\times m^{n-1}$ - quadrilateral labyrinth set. To prove $\mathcal{W}_n$  is an $m^n\times m^n$ - quadrilateral labyrinth set, it is enough to prove that $\mathcal W_n$ satisfies tree property, exit property and corner property.
\renewcommand{\theenumi}{\roman{enumi}}
\begin{enumerate}
\item $\mathcal{G}(\mathcal{W}_n)$ is a tree.\\
The $n^{th}$ stage is obtained by replacing each white quadrilateral in $(n-1)^{th}$ stage by a pattern similar to the first stage as given in definition. The exit property in the first stage and the connectedness of $(n-1)^{th}$ stage together gives the connectedness of  $\mathcal{G}(\mathcal{W}_n)$. To prove the graph is acyclic, suppose for the contrary that there is a cycle $C=\{a_0,a_1,...a_s\}$ in $\mathcal{G}(\mathcal{W}_n)$. For each $a\in\mathcal{W}_n$ let $t(a)$ be the white quadrilateral in $\mathcal{W}_{n-1}$, which contains $a$. Let $j_0=0$,$b_0=t(a_0)$ and $j_k=$ min $\{i : t(a_i)\neq b_{k-1}, j_{k-1}<i\leq s\}$, $b_k=t(a_{j_k})$ for $k\geq 1$. Choose $r$ minimal such that the set $\{i : t(a_i)\neq b_r, j_r<i\leq s\}$ is empty. For $i=1,2,...,r$ the vertex $b_{i-1}$ is a neighbour of $b_i$ in $\mathcal{G}(\mathcal{W}_{n-1})$.  The set $\{b_0,b_1,...,b_r\}$ cannot contain a cycle by the induction hypothesis. Also, if $r=0$, then all $a_i$'s are contained in $b_0$ and it will contradicts the fact that $\mathcal{G}(\mathcal{W}_1)$  is a tree. Hence $r\geq 1$. Thus the graph induced by $\mathcal{G}(\mathcal{W}_{n-1})$ on the set $\{b_0,b_1...,b_r\}$ is a tree with more than one vertex. This implies that the cycle $C$ returns to $a_0$ through the same side where it leaves $a_0$. But the exit property gives that $a_i=a_j$ for some $j\neq i$ ; $a_i,a_j\in\{a_0,a_1,...,a_s\}$, which is a contradiction. Hence $\mathcal{G}(\mathcal{W}_n)$ is a tree.
\item Since the exit property is satisfied in the $(n-1)^{th}$ and in the first stage, we have a unique exit in each side of the quadrilateral in both case. This ensures the exit property in the $n^{th}$ stage.
\item Upto the $(n-1)^{th}$ stage, corner property is satisfied. The corner in the $(n-1)^{th}$ stage if it exist will only contribute to the corner in the $n^{th}$ stage. So the corner property also clearly satisfied.
\end{enumerate}
\end{proof}
\begin{rem}
For a quadrilateral labyrinth set $\mathcal W_1$, it is shown that $\mathcal G(\mathcal W_n)$ is connected. Hence $L_n$ is connected for any $n\ge 1$. Thus $\{L_n\}$ is a decreasing sequence of nonempty compact connected sets. 
\end{rem}
\begin{defn}
If $\mathcal W_1$ is a quadrilateral labyrinth set, then the limit set\\ $\mathcal L_{\infty}=\bigcap\limits_{n=1}^{\infty}L_n=\bigcap\limits_{n=1}^{\infty}\bigcup\limits_{W\in\mathcal{W}_n}{W}$ is called the quadrilateral labyrinth fractal.\\ Figures $2,3,$ and $4$ show some examples of quadrilateral labyrinth fractals.
\end{defn}
\begin{figure}[tb]

\label{image1}

\includegraphics[width=\linewidth]{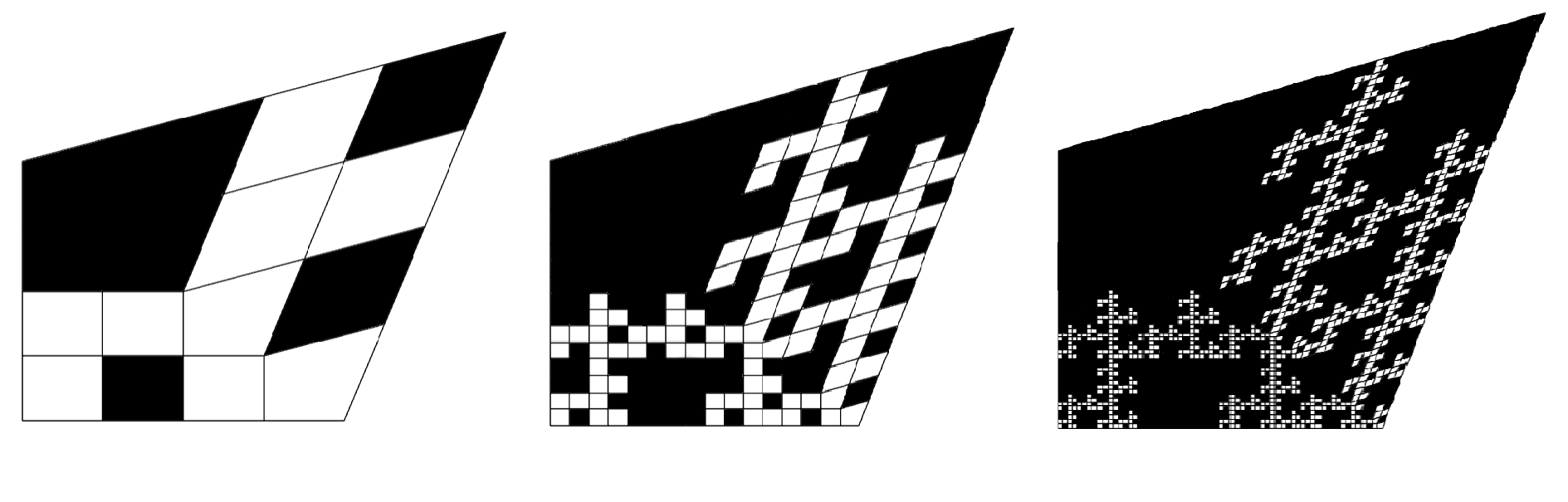}
\includegraphics[width=\linewidth]{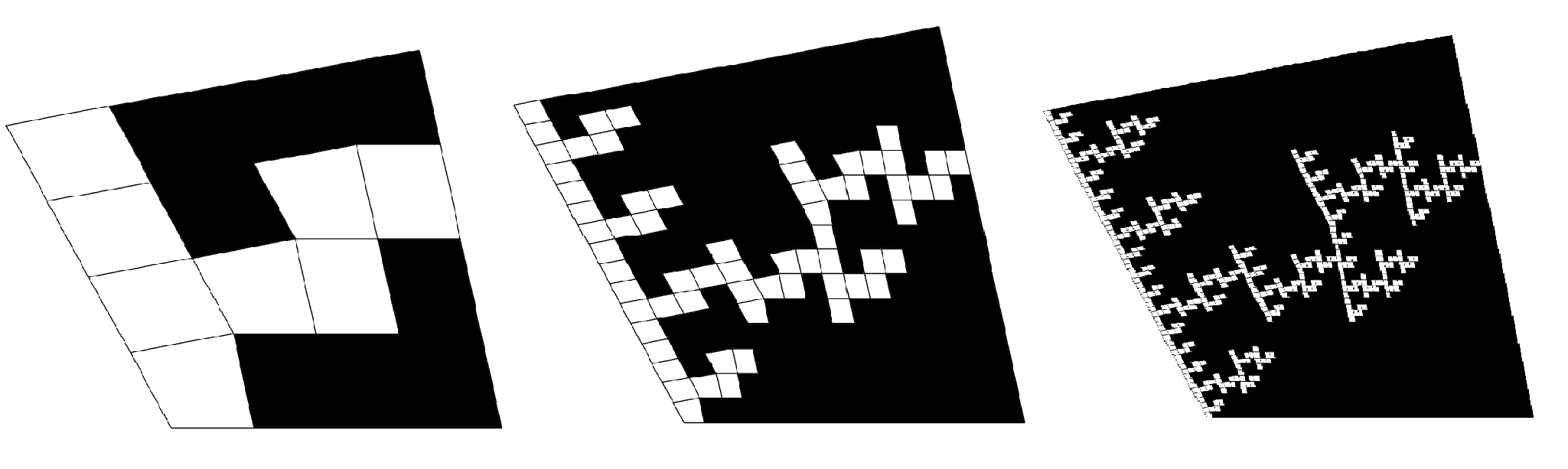}
\caption{First three stages of two different $4\times 4$-quadrilateral labyrinth fractals}
\end{figure}
\begin{figure}[tb]

\label{image1}

\includegraphics[width=\linewidth]{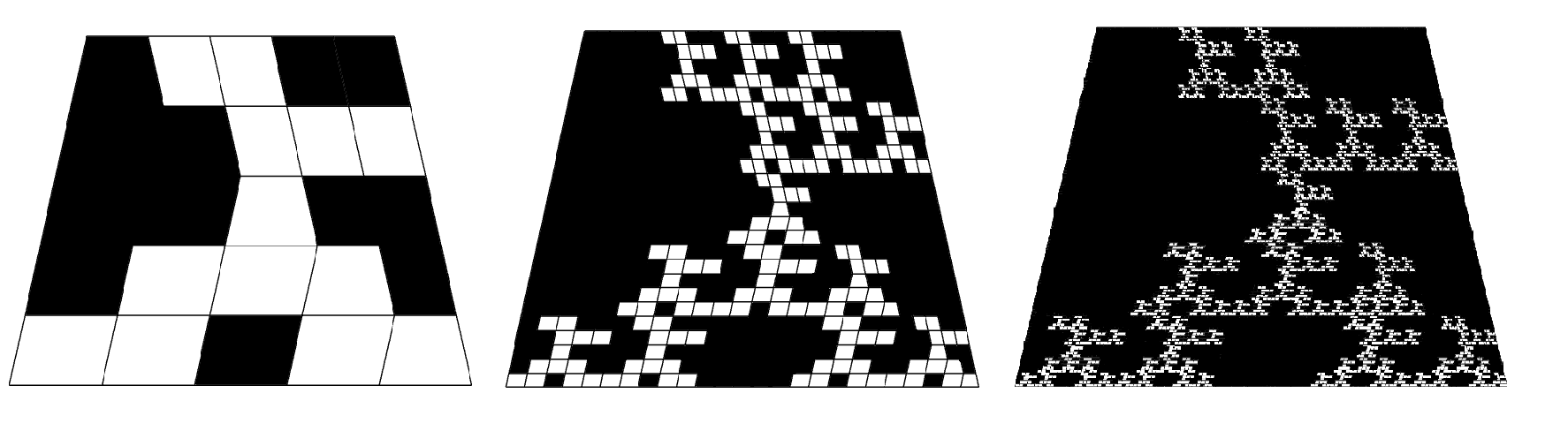}
\caption{First three stages of a $5\times 5$-quadrilateral labyrinth fractal}
\end{figure}
\begin{figure}[tb]

\label{image1}

\includegraphics[width=\linewidth]{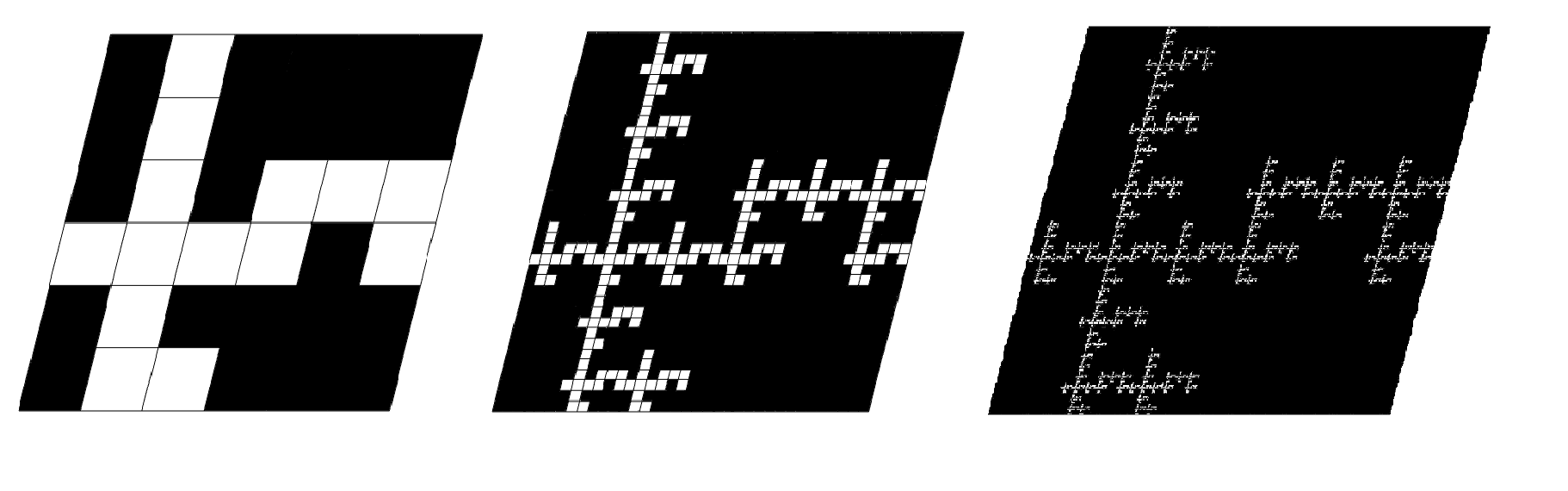}
\caption{First three stages of a $6\times 6$-quadrilateral labyrinth fractal}
\end{figure}
\begin{prop}
Let $\mathcal{W}_1$ be a quadrilateral labyrinth set in the quadrilateral $Q$. Then for any integer $n\ge 1$, $\mathcal W_n=\{P_{W_1^1,W_1^2,...,W_1^n}(Q) : W_1^i\in\mathcal W_1\forall i=1,2,...,n\}$, where, $P_{W_1^1,W_1^2,...,W_1^n}=P_{W_1^1}\circ P_{W_1^2}\circ ...\circ P_{W_1^n}$.
\end{prop}
\begin{proof}
The proof is by induction. For $n=1$, it is clear that \\$\mathcal W_1=\{P_{W_1}(Q) : W_1\in\nolinebreak\mathcal W_1\}$. So the result is true for $n=1$. Suppose the result holds upto $n-1$. $i.e.$, $\mathcal W_{n-1}=\{P_{W_1^1,W_1^2,...,W_1^{n-1}}(Q) : W_1^i\in\mathcal W_1, \forall i=1,2,...,n-1\}$.
Let $\mathcal V_n=\{P_{W_1^1,W_1^2,...,W_1^n}(Q) : W_1^i\in\mathcal W_1\forall i=1,2,...,n\}$. It is enough to prove that $\mathcal W_n=\mathcal V_n$, where $\mathcal W_n$ is given as $\mathcal{W}_n=\{P_{W_{n-1}}(W_1) : W_1\in\mathcal{W}_1,W_{n-1}\in\mathcal{W}_{n-1}\}$.
Choose  an element $P_{W_{n-1}}(W_1)\in \mathcal W_n$. Then  $P_{W_{n-1}}(W_1)= P_{W_{n-1}}\circ P_{W_1}(Q)\in \mathcal W_n$.\\
By induction hypothesis $W_{n-1}=P_{W_1^1,W_1^2,...,W_1^{n-1}}(Q)$ for $W_1^i\in\mathcal W_1, \forall i=1,2,...,n-~1$.\\
Hence,
\begin{equation*}
\begin{split}
P_{W_{n-1}}(W_1)&= P_{W_{n-1}}\circ P_{W_1}(Q)\\
&=P_{P_{W_1^1,W_1^2,...,W_1^{n-1}}(Q)}\circ P_{W_1}(Q)\\
&=P_{W_1^1}\circ P_{W_1^2}\circ ...\circ P_{W_1^{n-1}}\circ P_{W_1}(Q)\text{(Using Proposition $2.2$ recursively)}
\end{split}
\end{equation*}
Thus $\mathcal W_n\subseteq V_n$\\
Note that if $|\mathcal W_1|=k$, then clearly $|\mathcal W_n|=|\mathcal V_n|=k^n$. Hence $\mathcal W_n=\mathcal V_n$
\end{proof}

\begin{prop}
For all $n\ge 1$, $L_{\infty}=\bigcup\limits_{W_n\in\mathcal W_n}P_{W_n}(L_{\infty})$
\end{prop}
\begin{proof}
Using continuity of $P_{W_n}$ and Proposition $3.3$, it can be easily shown that $P_{W_n}(L_{\infty})=W_n\cap L_{\infty}$, for any $n\ge 1$.
Hence,
\begin{equation*}
 \bigcup\limits_{W_n\in\mathcal W_n}P_{W_n}(L_{\infty})=\bigcup\limits_{W_n\in\mathcal W_n}(W_n\cap L_{\infty})=L_{\infty}\cap\bigcup\limits_{W_n\in\mathcal W_n}W_n=L_{\infty}\cap L_n=L_{\infty}
\end{equation*}
\end{proof}
\section{TOPOLOGICAL PROPERTIES}
This section deals with the topological properties of labyrinth set and labyrinth fractal. The section also studies different connectedness properties in $Q\setminus L_n$ and $Q\setminus L_{\infty}$.
\begin{prop}
If $\mathcal W\subset\mathcal S_m$ is a set of white quadrilaterals such that the associated graph $\mathcal{G}(\mathcal W)$ is a tree, then from every black quadrilateral in $\mathcal B=\mathcal S_m\setminus \mathcal W$, there is a path in $\mathcal{G}(\mathcal B)$ to a border quadrilateral.
\end{prop}
\begin{proof}
Suppose for the contrary that there exist a black quadrilateral $B$ in $\mathcal G(\mathcal B)$ such that there does not exist a path from $B$ to any border quadrilateral. Since $\mathcal G(\mathcal W)$ is a tree, choose $W\in\mathcal W$ such that $W$ has only one neighbour in $\mathcal G(\mathcal W)$. WLOG, suppose this unique neighbour of $W$ lies at the right side of $W$. Let $\mathcal W'=\mathcal W\setminus\{W\}$ and $\mathcal B'=\mathcal B\cup\{W\}$ and consider the graphs $\mathcal G(\mathcal W')$ and $\mathcal G(\mathcal B')$.\\
Claim 1: There does not exist a path from $B$ to a border quadrilateral in $\mathcal G(\mathcal B')$.\\
Suppose the claim is not true. $i.e.,$ there exist a path from $B$ to one of the border quadrilateral in $\mathcal G(\mathcal B')$. Thus the connectivity component of $B$ in $\mathcal G(\mathcal B')$ should contain $W$, otherwise it will contradict the assumption.\\
Case 1:  If $W$ is a border quadrilateral, then it can be in top row, left row or bottom row. But in each of these cases, all quadrilaterals, except the quadrilateral on right side of $W$, having intersection with $W$ along a side of $W$, is a black quadrilateral and all these quadrilaterals are in the connectivity component of $B$ in $\mathcal G(\mathcal B')$. If $W$ is in top (or bottom) row, then the left and bottom (or top) neighbours of $W$, which are black both in $\mathcal G(\mathcal W)$ and $\mathcal G(\mathcal W')$ ensures path from $B$ in $\mathcal G(\mathcal B)$ to a border quadrilateral. Now if $W$ is a quadrilateral in the left column, then the top and bottom black neighbours of $W$ ensures a path from $B$ to a border square in $\mathcal G(\mathcal B)$. So in any case, the connectivity component of $B$ in $\mathcal G(\mathcal B)$ contains a border square and hence a contradiction to the assumption.\\
Case 2: If $W$ is not a border quadrilateral, and there exist a path from $B$ to border quadrilateral in $\mathcal G(\mathcal B')$ ensures a path from $W$ to a border quadrilateral. But the quadrilaterals on top, bottom and left side of $W$ are black and are in the connectivity component of $B$ in both $\mathcal G(\mathcal B)$ and $\mathcal G(\mathcal B')$. Hence in $\mathcal G(\mathcal B)$ also, a path exist from $B$ to  a border quadrilateral, which is again a contradiction.\\
Hence there does not exist a path from $B$ to a border quadrilateral in $\mathcal G(\mathcal B')$. \\
Again choose $W'\in \mathcal W'$ such that $W'$ has only one neighbour in $\mathcal G(\mathcal W')$ and repeat the process until no such white quadrilaterals having only one white neighbour exist. So all the white quadrilateral, if it exist, have degree greater than one and it will result in a cycle in $\mathcal G(\mathcal W)$, which is not possible. Hence all quadrilaterals are black. But it agains leads to a contradiction to the Claim 1.
\end{proof}
\begin{defn}
 Let $X$ is a topological space and $x_0,x_1\in X$. Then an arc in $X$ from $x_0$ to $x_1$ is a continuous function $\gamma:[0,1]\to X$ such that $\gamma(0)=x_0$ and $\gamma(1)=x_1$.
\end{defn}
\begin{defn}
A dendrite is a locally connected continuum that contain no simple closed curve and a continuum is a nonempty compact connected Hausdorff space.
\end{defn}
\begin{prop}
If $x$ is a point in $Q\setminus L_n$, then there is an arc $a\subseteq Q\setminus L_{n+1}$ between $x$ and a point in the boundary of $Q$.
\end{prop}
\begin{proof}
Since $x\in Q\setminus L_n$, there exist a black quadrilateral $B\in \mathcal B_n$ such that $x\in B$. By Proposition $4.1$, there exist a path $B_0B_1...B_k$ in $\mathcal B_n$ from $B=B_0$ to a border quadrilateral $B_n$. Construct an arc $a'$ from $x$ to centre of $B_0$ and then from centre of $B_i$ to $B_{i+1}$ for $i=0,1,...,n-1$ and then from $B_n$ to the boundary of $Q$. If $B_i$ and $B_{i+1}$ have a common side for all $i=1,2,...,n-1$, then the arc $a'$ can be taken as the required arc. If $B_i$ and $B_{i+1}$ are neighbours along a common vertex $c$ only, then $c\in L_n$ and so $c\not\in Q\setminus L_n$. In this case $c$ may or may not belongs to $Q\setminus L_{n+1}$. If $c\not\in Q\setminus L_{n+1}$, then modify the arc $a'$ to $a$ as follows: Choose the black quadrilateral $B'\in \mathcal B_{n+1}$, $B'\subseteq W'\cup W''$, where, $W',W''\in \mathcal W_n$ and have one vertex as $c$. Such a black quadrilateral is ensured by the corner property. From $B_i$ to $B_{i+1}$ use $B'$ for the arc $a$ by removing $c$ from $a'$. Hence $a$ will be a required arc in $Q\setminus L_{n+1}$.
\end{proof}
\begin{coro}
If $x$ is a point in $Q\setminus L_{\infty}$, then there is an arc $a\subseteq Q\setminus L_{\infty}$ between $x$ and a point in the boundary of $Q$.
\end{coro}
\begin{theo}
$L_{\infty}$ is a dendrite.
\end{theo}
\begin{proof}
$L_{\infty}=\bigcap\limits_{n=1}^{\infty}L_n$ and $\{L_n\}$ is a decresing sequence of nonempty connected compact sets. Clearly $L_{\infty}$ is nonempty, since it is the intersection of a decreasing sequence of nonempty compact sets. Since each $L_n$ is compact, they are closed and bounded. Hence $L_{\infty}$ is also closed and bounded. So $L_{\infty}$ is compact. Also $L_{\infty}$ is a subspace of the Hausdorff space $\mathbb{R}^2$, and hence $L_{\infty}$ is Hausdorff. Suppose $L_{\infty}$ is not connected. Then there exist two disjoint open sets $U$ and $V$ such that $L_{\infty}=U\cup V$.\\ Consider the sets $F_i=L_i\setminus L_{\infty}=L_i\setminus U\cup V$ for $i=1,2,3,...$. Then $\{F_i\}$ is a decreasing sequence of compact sets and $\bigcap\limits_{i=1}^{\infty}F_i=\phi$. Hence there exist an $n\in \mathbb{N}$ such that $F_n=\phi$ and so $L_n\subseteq L_{\infty}$. Also $L_{\infty}\subseteq L_n$. Thus $L_n=L_{\infty}=U\cup V$ implies $L_n$ is disconnected, which is not true.   
Thus $L_{\infty}$ should be connected. Hence $L_{\infty}$ is a continuum. For any $\epsilon>0$, we can find an $n$ such that diameter of $W_n$ is less than $\epsilon$ for all $W_n\in\mathcal W_n$ and by Proposition $3.4$,  $L_{\infty}=\bigcup\limits_{W_n\in\mathcal W_n}P_{W_n}(L_{\infty})$. Here $P_{W_n}(L_{\infty})$ is a continuum for any $W_n\in\mathcal W_n$, since $L_{\infty}$ is a continuum.  Hence $L_{\infty}$ is the finite union of continua of diameter less than $\epsilon$.  Then by Hahn-Mazurkiewicz-Sierpinski theorem (Topology; K.Kuratowski\cite{7}), $L_{\infty}$ is locally connected.\\
Now suppose $L_{\infty}$ contains a simple closed curve $c$. Then by Jordan curve theorem, the interior of $c$ is nonempty and open. For sufficiently large $n$ there exist a black quadrilateral $B$ of order $n$ in interior of $c$. So there exist an arc from a point of $B$ to the boundary of the quadrilateral(By Corollary $4.1$). This arc will cross $c$, which is not possible, since $c\subseteq L_{\infty}$. So $L_{\infty}$ does not contain a simple closed curve.
\end{proof}
\begin{theo}
The interior of $L_{\infty}$ is empty.
\end{theo}
\begin{proof}
Suppose not. Then there exist a point $x\in Int(L_{\infty})$ and an $r>0$ such that ball with cente $x$ and radius $r$, say $B(x,r)$, is contained in $L_{\infty}$. Choose an $\epsilon$ such that $o<\epsilon<r$ and $\overline {B(x,\epsilon)}\subseteq B(x,r)$. Then clearly the circle with centre $x$ and radius $\epsilon$ is a simple closed curve contained in $L_{\infty}$. It is a contradiction.
\end{proof}
\begin{coro}
$L_{\infty}$  is a first category subset in $\mathbb{R}^2$.
\end{coro}
\begin{proof}
Since $L_{\infty}$ is closed, interior of closure of $L_{\infty}$ is same as the interior of $L_{\infty}$ and it is empty. Thus $L_{\infty}$ is nowhere dense subset of $\mathbb{R}^2$.\\
Consider closed intervals in $\mathbb{R}$ with rational endpoints and then consider a strip on $\mathbb{R}^2$ corresponding to each of these intervals. i.e, if $[r_1,r_2]$ is a interval with $r_1,r_2\in \mathbb{Q}$, consider the strip $[r_1,r_2]\times y-axis$. Collection of all such strip is, say $\{A_i\}$, is countable. Let $B_i=A_i\cap L_{\infty}$ , then $\{B_i\}$ is countable and clearly $L_{\infty}=\bigcup\limits_{i=1}^{\infty}B_i$ and each $B_i$ is nowhere dense, since they are subsets of nowhere dense set $L_{\infty}$.
Hence $L_{\infty}$ is of first category subset in $\mathbb{R}^2$. 
\end{proof}
\begin{rem}
Let $L_{\infty}$ as a subspace of $\mathbb{R}^2$ with induced topology. Then $L_{\infty}$ is compact and Hausdorff and hence a Baire's space. Since every Baire's space is of second category, $L_{\infty}$ is a second category subspace. (Topology; J.R. Munkres\cite{8})
\end{rem}

\begin{theo}
$Q\setminus L_{\infty}$ is not path connected.
\end{theo}
\begin{proof}
Choose an exit $W_L$ of $\mathcal W_1$ such that $W_L$ is not a corner quadrilateral. Such an exit always exist by Proposition $3.1$ and corner property. WLOG suppose that $W_L$ is the left exit and let $P$ be a path from $W_L$ to $W_R$, where $W_R$ is the right exit of $\mathcal W_1$, which is clearly not a corner quadrilateral. Let $a=L_{\infty}\cap P$. Then $a$ is a path from left exit to right exit in $L_{\infty}$. Since both bottom-left corner and bottom-right corner cannot be white in $\mathcal W_1$, choose the black quadrilateral, say $B$, from these two corners. In a similar way choose a black quadrilateral, say $B'$, from top-left corner and top-right corner of $\mathcal W_1$. Let $x\in B$ and $y\in B'$. Now, if there exist a path from $x$ to $y$, then it should intersect with $a$, which is not possible, since $a\subseteq L_{\infty}$. So there does not exist a path from $x$ to $y$ in $Q\setminus L_{\infty}$ and so $Q\setminus L_{\infty}$ is not path connected. 
\end{proof}
\pagebreak
\begin{coro}
$Q\setminus L_{\infty}$ is not connected.
\end{coro}
\begin{proof}
Claim: $Q\setminus L_{\infty}$ is locally path connected.\\
Let $x\in Q\setminus L_{\infty}$,  and $U$ is any open set containing $x$, then for sufficiently small $r$ there exist an open ball containing $x$ and contained in $U$ in subspace topology and clearly it will be path connected. Hence  $Q\setminus L_{\infty}$ is locally path connected.\\
Since every locally path connected space is connected if and only if the space is path connected, $Q\setminus L_{\infty}$ is not connected. (By Theorem 4.3)
\end{proof}
\begin{coro}
$Q\setminus L_n$ is not path connected and not connected for any $n\ge 1$.
\end{coro}
\begin{theo}
If $\mathcal W_1$ is an $m\times m$ - quadrilateral labyrinth set in $Q$, $m\ge 4$, such that $\mathcal W_1$ contains no border quadrilateral except the exits, and $L_{\infty}$ is the quadrilateral labyrinth fractal generated from $\mathcal W_1$, then the number of connected components of $Q\setminus L_{\infty}$ is $4$.
\end{theo}
\begin{proof}
Claim 1: $\mathcal W_1$ cannot have a corner quadrilateral.\\
Suppose there exist a corner quadrilateral, say $W_0$ in $\mathcal W_1$. Let $W_0$ be the top-left corner. Then it should be either left-exit or top-exit. WLOG suppose $W_0$ is the top-exit. Since $\mathcal W_1$ is connected, there exit a border quadrilateral, say $W_1$, which is a neighbour of $W_0$. Since $W_1$ is a border quadrilateral, it should be an exit. Also $W_1$ cannot be a top exit, so $W_1$ is a left exit. Since top-left corner is the top-exit, bottom-left corner should be the bottom exit. Same as argument above there exist a left exit, say $W_2$, which is a neighbour of bottom left corner. Since the left exit is unique, only possibility is $W_1=W_2$. But this case holds only if $m=3$. So $\mathcal W_1$ contains no corner quadrilateral. 
Hence each vertex of $Q$ is in $Q\setminus L_{\infty}$.\\
Claim 2: A connected component of $Q\setminus L_{\infty}$ does not contain more than one vertex of $Q$. \\
Suppose not. Let $K$ be a connected component of $Q\setminus L_{\infty}$ which contains two corner vertices. WLOG assume top-left vertex and top-right vertex are contained in $K$. Since $Q\setminus L_{\infty}$ is locally path connected, the path connected components and connected components of $Q\setminus L_{\infty}$ coincides. Thus $K$ is a path connected component containing both top-left vertex and top-right vertex and there exist a path in $K$ between these two vertices. But this path will intersect with the path from top exit to bottom exit as shown in Theorem $4.1$, and it is not possible. Hence a connected component of $Q\setminus L_{\infty}$ contains atmost one vertex of $Q$.\\
Claim 3: Each connected component of $Q\setminus L_{\infty}$ contains atleast one vertex. \\
From Corollary $4.1$, it is clear that there exist a path from any point of $Q\setminus L_{\infty}$ to one of the boundary point of $Q$. Also from any border point in $Q\setminus L_n$, there exist a path to one of the vertices of $Q$, if not, it will be a contradiction to the hypothesis of theorem. Hence each components of $Q\setminus L_n$ contains atleast one vertex of $Q$.\\
Hence corresponding to each vertex there is a unique connected component in $Q\setminus L_{\infty}$ and so  the number of connected components of $Q\setminus L_{\infty}$ is $4$.
\end{proof}
\section{CONCLUSIONS}
In this paper, labyrinth fractals on a convex quadrilateral are investigated. On any convex quadrilateral, labyrinth fractals are defined in the same manner as in a  square or a triangle. However, the labyrinth fractal is not self-similar as the construction of smaller quadrilaterals lacks self-similarity property. 
\par
 The topological properties of both quadrilateral labyrinth fractal and its complement on the quadrilateral are also studied in this paper. It is found that quadrilateral labyrinth fractal is the first subcategory subset in $\mathbb{R}^2$, and it is a dendrite. In future, many properties related to a fractal such as  various dimensions shall be investigated.

\bibliographystyle{amsplain}

\end{document}